\documentclass[a4paper,12pt]{amsart}
\usepackage{amssymb}
\usepackage{amsmath}
\usepackage{graphics}
\usepackage{soul}
\usepackage{xcolor}

\usepackage[dvips]{graphicx}
\newtheorem{thm}{Theorem}

\def\C{{\mathbb C}}

\def\D{{\mathbb D}}
\def\R{{\mathbb R}}
\newcommand{\RR}{{\mathbb R}}

\def\be{\begin{equation}}
\def\ee{\end{equation}}

\def\C{{\mathbb C}}

\def\D{{\mathbb D}}
\def\R{{\mathbb R}}

\makeatletter

\begin{document}

\title[On the Conformal Lens Maps]{On the Conformal Lens Map}

\thanks{The
authors were partially supported by Fondecyt Grants \# 1190756.
\endgraf  {\sl Key words:} {Convex Mappings, Lens Maps.}
\endgraf {\sl 2020 AMS Subject Classification}. Primary: 30C45, 30C99; \,
Secondary: 31C05.}

\author{Vicente Ahumada}
\address{Facultad de Ingenier\'ia y Ciencias\\
Universidad Adolfo Ib\'a\~nez\\
Av. Padre Hurtado 750, Vi\~na del Mar, Chile.}
\email{vahumada@alumnos.uai.cl}

\author{Eduardo Behm}
\address{Facultad de Ingenier\'ia y Ciencias\\
Universidad Adolfo Ib\'a\~nez\\
Av. Padre Hurtado 750, Vi\~na del Mar, Chile.}
\email{ebehm@alumnos.uai.cl}

\author{Rodrigo Hern\'andez }
\address{Facultad de Ingenier\'ia y Ciencias\\
Universidad Adolfo Ib\'a\~nez\\
Av. Padre Hurtado 750, Vi\~na del Mar, Chile.}
\email{rodrigo.hernandez@uai.cl}

\author{Dubalio P\'erez }
\address{Facultad de Ingenier\'ia y Ciencias\\
Universidad Adolfo Ib\'a\~nez\\
Av. Padre Hurtado 750, Vi\~na del Mar, Chile.}
\email{duperez@alumnos.uai.cl}


\begin{abstract} It is well-known that lens maps are convex mappings defined in the unit disc to itself. In this brief note, we show that these mappings are convex of order $\alpha>0$, and starlike of order $\beta>0$, and establish the precise orders in terms of opening angle of the lens map.  
\end{abstract}

\maketitle

\section{Introduction}

The study of conformal functions has been very extensive in the last century and many applications can be described by these functions. For instance, the author in \cite{FS} uses these kind of mappings to understand the visual field, which is represented in the brain by mappings which are, at least approximately, conformal. Since conformal mappings preserve essential visual information, \cite{Baricco} states that they are used for the solutions of 2D potential problems on conservative fields, while they are also used for face recognition and texture processing in computer graphics. For more detail conformal mappings, see \cite{NEHARI_CONFORMAL}. In these brief notes, we will study some geometric aspects of what we call \textit{lens maps} such as the order of starlikeness and convexity.\\

Let $\zeta$ be the complex function defined in the unitary disc, where $z\in\mathbb{D}$: $\zeta(z)=(1+z)/(1-z)$ as a M\"obius transformation that maps the disc into a half plane. Now, we can define the \textit{lens function} with the parameter $\alpha\in(0,1]$ as
\begin{equation}\label{def-lens}
    \ell_{\alpha}(z)=\frac{\xi-1}{\xi +1}  =\frac{\zeta^\alpha-1}{\zeta^\alpha +1} = \frac{\left(\dfrac{1+z}{1-z}\right)^{\alpha}-1}{\left(\dfrac{1+z}{1-z}\right)^{\alpha} +1}.
\end{equation} To see some application of lens maps read \cite{ChHM}.

Let $\mathcal{C}$ be the class of holomorphic and univalent functions  normalized by the conditions $f(0)=0$ and $f'(0)=1$ that transform the unitary disc into a convex region in the complex plane. Then, an important characterization  of convex functions is that $f\in\mathcal{C}$ if and only if
\begin{equation}\label{convex}
    \operatorname{Re}\left\{1+z\frac{f''(z)}{f'(z)}\right\}>0.
\end{equation} For more details on convex mappings, refer to\cite{DUREN_UF,GK,NEHARI76}. One of the many extensions of the class of convex functions was provided by Robertson in \cite{ROBERTSON}: $f$ is a convex function of order $\alpha$, with $0\leq\alpha<1$, if
\begin{equation*}\label{convex_alpha}
    \RR\left\{1+z\,\dfrac{f''}{f'}(z)\right\}>\alpha.
\end{equation*}
We denote the class of convex function of order $\alpha$ as $\mathcal{C}_\alpha$. Observe that $\mathcal{C}_0=\mathcal{C}$ and $\mathcal{C}_\alpha\subset \mathcal{C}$ for all $0\leq\alpha<1$. Robertson \cite{ROBERTSON} observed that functions $f$ in $\mathcal C_\alpha$, have the following geometric property: the ratio of the angle between adjacent tangents of the unit circle over the
angle between the corresponding tangents in the image of $f$ is less than $1/\alpha$. Hence the closer $\alpha$ is to 1 the ``rounder'' is the image. For instance, segments in the boundary of $f(\D)$ are prohibited as soon as $\alpha > 0$.

\section{Results}

It is a well-known result that lens maps are bounded convex mappings. However, it is not easy to see that actually is a convex mapping with order $\beta>0$. This is the main purpose of this notes. In fact, lens maps are a starlike maps for every $\alpha\neq0$ and $\alpha\in [-2,2]$, and convex when $\alpha\in[-1,1]$. In this manuscript we find the order of starlikness and convexity of the lens maps in terms of the parameter $\alpha$.

\subsection{Starlikeness}

Let $f:\D\to\C$ be a locally univalent analytic mapping normalized as before, such that the range of $f(\D)$ is a starlike domain with respect to the origin. A well known analytical characterization of starlikeness is given by $$\mbox{Re}\left\{z\frac{f'}{f}(z)\right\}>0.$$ See \cite{DUREN_UF,GK}. For lens map $\ell_\alpha$, it is easy to see that $$z\frac{\ell_\alpha'}{\ell_\alpha}(z)=\frac{4\alpha z}{1-z^2}\cdot \frac{\left(\dfrac{1+z}{1-z}\right)^\alpha}{\left(\dfrac{1+z}{1-z}\right)^{2\alpha}-1}=\dfrac{4\alpha z(1-z^2)^{\alpha-1}}{(1+z)^{2\alpha}-(1-z)^{2\alpha}}.$$ Thus, if $\alpha\in(1,2]$ this quantity tends to zero when $z$ goes to $\pm 1$, then $$\inf_{z\in \D}\mbox{Re}\left\{z\frac{\ell_\alpha'}{\ell_\alpha}(z)\right\}=0,$$ which means that $\ell_\alpha$ is starlike map with order 0. Since $\ell_{\alpha}(z)=\ell_{-\alpha}(-z)$ then, for $\alpha\in[-2,-1)$, $$\inf_{z\in\D}\mbox{Re}\left\{z\frac{\ell_\alpha'}{\ell_\alpha}(z)\right\}=\inf_{w\in\D}\mbox{Re}\left\{w\frac{\ell_{-\alpha}'}{\ell_{-\alpha}}(w)\right\}=0.$$ Now, for $\alpha\in[-1,1]$ the corresponding lens map is a convex mapping, therefore is a starlike map with order $1/2$ at least. This is the fact in the following theorem.

\begin{thm} For $\alpha\in[-1,1]$, $\ell_\alpha$ is a starlike mapping with order $\beta$, given by $$\beta=\dfrac{\alpha}{\sin\left(\dfrac{\pi}{2}\alpha\right)}.$$   
\end{thm}

\begin{proof} A direct calculation shows that, for $|z|=1$ with $z\neq 1$ we have that  $$\mbox{Re}\left\{z\frac{\ell_\alpha'}{\ell_\alpha}(z)\right\}=2\alpha\cdot\mbox{Im}\left\{\dfrac{z}{1-z^2}\right\}\left(\dfrac{\mbox{Im}\{\xi\}}{|\xi-1|^2}+\dfrac{\mbox{Im}\{\xi\}}{|\xi+1|^2}\right).$$ By the symmetry of the lens maps, we consider $z=e^{it}$ with $t\in(0,\pi)$ and its follows that $$\operatorname{Im}\left\{\frac{z}{1-z^2} \right\} =\frac{2\operatorname{Im}\{z \}}{\lvert 1-z^2 \rvert ^2}=\frac{1}{2\sin(t)},$$ and 
$$\frac{\operatorname{Im}\{\xi\}}{\lvert 1 + \xi \rvert^2} = \frac{b(t)\sin(\frac{\pi}{2}\alpha)}{1+2b(t)\cos(\frac{\pi}{2}\alpha)+b(t)^2},$$ and $$\frac{\operatorname{Im}\{\xi\}}{\lvert 1 - \xi \rvert^2} = \frac{b(t)\sin(\frac{\pi}{2}\alpha)}{1-2b(t)\cos(\frac{\pi}{2}\alpha)+b(t)^2}.$$ where $b(t)=|\xi(t)|=(\sin(t)/(1-\cos(t)))^\alpha=((1+\cos(t))/\sin(t))^\alpha$. Thus \begin{equation}\label{op.starlike}\mbox{Re}\left\{z\frac{\ell_\alpha'}{\ell_\alpha}(z)\right\}=\alpha\sin\left(\dfrac{\pi}{2}\alpha\right)\frac{b(t)}{\sin(t)}\left(\dfrac{2(b^2(t)+1)}{(b^2(t)+1)^2-4b^2(t)\cos^2\left(\dfrac{\pi}{2}\alpha\right)}\right),\end{equation} then $$\beta=\alpha\sin\left(\dfrac{\pi}{2}\alpha\right)\cdot\inf_{t\in(0,\pi)}\frac{b(t)}{\sin(t)}\left(\dfrac{2(b^2(t)+1)}{(b^2(t)+1)^2-4b^2(t)\cos^2\left(\dfrac{\pi}{2}\alpha\right)}\right).$$ The relationship between $b(t)$ and $\sin(t)$ is given by \begin{equation}\label{relation-b-sin}\sin(t)=\dfrac{2 b(t)^{\frac{1}{\alpha}}}{1+b(t)^{\frac{2}{\alpha}}},\quad t\in(0,\pi).\end{equation} Therefore, for any $t\in(0,\pi)$, let $x=b(t)$ be the real number that decrease from $\infty$ to $0$, then the right hand in (\ref{op.starlike}) is given by $$\alpha\sin\left(\dfrac{\pi}{2}\alpha\right)\cdot\frac{(1+b^{2/\alpha}(t))(b^2(t)+1)}{b^{1/\alpha-1}(t)\left((b^2(t)+1)^2-4b^2(t)\cos^2\left(\dfrac{\pi}{2}\alpha\right)\right)},$$ which goes to infinity when $b(t)$ tends to 0, and when $b(t)$ tends to $\infty$. Hence, the infimum is attained for some $t\in(0,\pi)$, therefore it is a minimum named $m(\alpha)$. Thus $$\beta =\alpha \sin\left(\frac{\alpha\pi}{2}\right)\cdot m(\alpha).$$ 

Let $g(x):(0,\infty)\to\R$ the real functions given by $$g(x)=\frac{x\left(1+x^2\right) }{\left(1+x^2\right)^2-4 x^2 \cos^2\left(\frac{\alpha \pi }{2}\right)}\cdot \frac{\left(1+x^{2/\alpha}\right)}{x^{\frac{1}{\alpha}}}=f(x)\cdot\frac{\left(1+x^{2/\alpha}\right)}{x^{\frac{1}{\alpha}}} ,$$ then 
$$ g'(x)=f'(x)\cdot\frac{\left(1+x^{2/\alpha}\right)}{x^{\frac{1}{\alpha}}}+f(x)\cdot \frac{2(x^{1/\alpha-1}-x^{3/\alpha-1})}{\alpha(1+x^{2/a})^2}.$$ Here, $$f'(x)=\frac{2\left((3x^2+1)\left(\left(1+x^2\right)^2-4 x^2 \cos^2\left(\frac{\alpha \pi }{2}\right)\right)-(x^3+x)(4x(x^2+1)-8x\cos^2(\frac{\pi \alpha}{2}))\right)}{\left(\left(1+x^2\right)^2-4 x^2 \cos^2\left(\frac{\alpha \pi }{2}\right)\right)^2}$$ A straightforward calculation gave as the $$f'(x)=\dfrac{2(1-x^2)\left((1+x^2)^2+4x^2\cos^2(\frac{\pi \alpha}{2} )\right)}{\left(\left(1+x^2\right)^2-4 x^2 \cos^2\left(\frac{\alpha \pi }{2}\right)\right)^2}.$$
Thus, evaluating at $x=1$ we have that $$g'(1)=2f'(1)=0.$$ Moreover, for all $x\in(0,1)$ we have that $g'(x)<0$ and for all $x>1$ it is follows that $g'(x)>0$, therefore $g(x)$ has a unique critical point at $x=1$ for positive real line, which is a minimum. Thus $$g(x)\geq g(1)=\frac{2}{4-4\cos^2(\frac{\pi \alpha}{2})}\cdot 2=\frac{1}{\sin^2(\frac{\pi \alpha}{2})}=m(\alpha),$$ so the proof is finished.
\end{proof}

\subsection{Convexity}
Since, $\ell_\alpha$ is an analytic convex mappings, the real part is an harmonic mappings with lower bound 0, therefore the minimum value of this real part is attained in some point in $|z|=1$, then $$\beta=\inf_{z\in\D}\mbox{Re}\left\{1+\frac{z \ell_a''}{\ell_a'}(z)\right\}=\min_{z\in\partial\D}\mbox{Re}\left\{1+\frac{z \ell_a''}{\ell_a'}(z)\right\}.$$

A direct calculation shows that 
\begin{equation}
\frac{\ell_a''}{\ell_a'}
=\frac{\xi^{\prime}}{\xi}+\frac{2 z}{1-z^2}-\frac{2 \xi^{\prime}}{\xi+1}=\frac{2 \alpha}{1-z^2}+\frac{2 z}{1-z^2}-\frac{4 \alpha \xi}{\left(1-z^2\right)(\xi+1)},
\end{equation}
thus

\begin{align*}
1+\frac{z \ell_a''}{\ell_a'}(z) &=1+\frac{2 z^2}{1-z^2}+\frac{2 \alpha z}{1-z^2}-\frac{4 \alpha z \xi}{\left(1-z^2\right)(1+z)} \\
&=\frac{1+z^2}{1-z^2}+\frac{2 \alpha z(1+\xi)-4 \alpha \xi z}{\left(1-z^2\right)(1+\xi)} \\
&=\frac{1+z^2}{1-z^2}+\frac{2 \alpha z-2 \alpha z \xi}{\left(1-z^2\right)(1+\xi)} \\
&=\frac{1+z^2}{1-z^2}-\frac{2 \alpha z}{1-z^2} \cdot \frac{\xi-1}{\xi+1}.
\end{align*} Therefore, taking the real part in the unit circle $|z|=1$ with $z\neq\pm 1$, we have that  
\begin{equation*}
\operatorname{Re}\left\{\frac{1+z^2}{1-z^2}-\frac{2 \alpha z}{1-z^2} \cdot \frac{\xi-1}{\xi+1}\right\}
=-2 \alpha\operatorname{Re}\left\{\frac{z}{1-z^2} \cdot \frac{\xi-1}{\xi+1}\right\}.
\end{equation*} Since $\operatorname{Re}\left\{z/(1-z^2)\right\} = 0 $ when $|z|=1$, then
\begin{equation*}
-2 \alpha\operatorname{Re}\left\{\frac{z}{1-z^2} \cdot \frac{\xi-1}{\xi+1}\right\} 
= 2\alpha\operatorname{Im}\left\{\frac{z}{1-z^2}\right\}\operatorname{Im}\left\{\frac{\xi -1}{\xi +1}\right\}
\end{equation*}
Then, we are searching for:
\begin{equation*}
\min_{z\in\partial\D}\left\{ 2\alpha\operatorname{Im}\left\{\frac{z}{1-z^2}\right\}\operatorname{Im}\left\{\frac{\xi -1}{\xi +1}\right\}\right\}.
\end{equation*} Is not difficult to see that, for all $|z|=1$ we have that $$\operatorname{Im}\left\{\frac{z}{1-z^2} \right\} = \frac{\operatorname{Im}\{z \}(1+|z|^2)}{\lvert 1-z^2 \rvert ^2} = \frac{2\operatorname{Im}\{z \}}{\lvert 1-z^2 \rvert ^2},$$ and $$\operatorname{Im}\left\{\frac{\xi -1}{\xi +1} \right\} = \operatorname{Im}\left\{\frac{(\xi-1)(\Bar{\xi} +1)}{\lvert \xi +1 \rvert ^2}\right\}= \frac{2\operatorname{Im}\{\xi \}}{\lvert \xi +1 \rvert ^2}.$$ Thus, an alternative version of the real part, for all $z\in\partial\D$, is given by $$\operatorname{Re}\left\{1+\frac{z \ell_a''}{\ell_a'}(z)\right\}=8\alpha \frac{\operatorname{Im}\{z\} \operatorname{Im}\{\xi\}}{\lvert \xi +1 \rvert ^2 \lvert 1 -z^2 \rvert ^2}.$$

The following theorem is the main results of this manuscript.

\begin{thm} $\ell_\alpha$ is a convex mapping with order $\beta$, given by
$$\beta=\dfrac{\alpha\sin\left(\dfrac{\pi}{2}\alpha\right)}{1+\cos\left(\dfrac{\pi}{2}\alpha\right)}.$$
\end{thm}

\begin{proof} Let $z = e^{it}$  and
$\xi(t) = \left(\dfrac{1+e^{it}}{1-e^{it}}\right)^\alpha$ ; $t \in [0,2\pi]$
\begin{equation}\label{def-f}
    f(t) = 2\alpha\operatorname{Im}\left\{\frac{e^{it}}{1-e^{2it}}\right\}\operatorname{Im}\left\{\frac{\xi(t) -1}{\xi(t) +1}\right\}.
\end{equation} Now, our effort will be to compute the minimum value of this positive function. Since $\zeta(t)=(1+e^{it})/(1-e^{it})=i\sin(t)/(1-\cos(t))$ and putting $\operatorname{Im}\{\xi(t)\}=b(t)\sin(\alpha\pi/2)$, we have that \begin{equation}\label{def-b}|\xi(t)| = \left | \frac{\sin (t)}{1-\cos (t)}\right|^\alpha, \quad b(t)=\left\{\begin{array}{ll}\quad\left(\frac{\sin (t)}{1-\cos (t)}\right)^\alpha, & t\in(0,\pi)\\[0.3cm]
-\left(\frac{-\sin (t)}{1-\cos (t)}\right)^\alpha, & t\in(\pi,2\pi).\end{array}\right.\end{equation} From where, for any $\alpha\in(0,1]$, $b(t)$ decrease from $\infty$ to 0 when $t$ runs from $0$ to $\pi$, and continuing been decreasing functions from 0 to $-\infty$ when $t$ goes from $\pi$ to $2\pi$. Since, imaginary parts are given by 
\begin{equation*}
    \operatorname{Im}\left\{\frac{z}{1-z^2} \right\} =\frac{2\operatorname{Im}\{z \}}{\lvert 1-z^2 \rvert ^2}=\frac{1}{2\sin(t)},
\end{equation*} and \begin{equation*}
\operatorname{Im}\left\{\frac{\xi(t) -1}{\xi(t) +1}\right\}=    \frac{2\operatorname{Im}\{\xi\}}{\lvert 1 + \xi \rvert^2} = \frac{2b(t)\sin(\frac{\pi}{2}\alpha)}{1+2b(t)\cos(\frac{\pi}{2}\alpha)+b(t)^2}.\end{equation*} Then $$f(t)=\frac{2\alpha b(t) \sin\left(\frac{\pi}{2} \alpha\right)}{1+2b(t) \cos \left(\frac{\pi}{2} \alpha\right)+b(t)^2}\cdot \frac{1}{\sin(t)}.$$ For $t\in(0,\pi)$, $b(t)$ satisfies that \begin{equation*}b(t)=\left(\dfrac{\sin(t)}{ 1 - \cos(t)}\right)^{\alpha}=\left(\dfrac{1+\cos(t)}{ \sin(t)}\right)^{\alpha}.\end{equation*}  Thus, if $t\to0^+$ then  $$f(t)=\frac{2\alpha b(t)\sin\left(\frac{\pi}{2} \alpha\right)}{1+2b(t) \cos \left(\frac{\pi}{2} \alpha\right)+b(t)^2}\cdot \frac{b(t)^{1/\alpha}}{(1+\cos(t))}\sim b(t)^{1/\alpha-1},$$ and $f(t)\to\infty$, since $b(t)$ goes to infinity. On the other hand, if $t\to\pi^-$, then $$f(t)=\frac{2\alpha\sin(t)^\alpha\sin\left(\frac{\pi}{2} \alpha\right)}{1+2b(t) \cos \left(\frac{\pi}{2} \alpha\right)+b(t)^2}\cdot \frac{1}{(1-\cos(t))^\alpha\sin(t)}\sim \dfrac{1}{\sin(t)^{1-\alpha}},$$ goes to infinity. Now, if $t\in(\pi,2\pi)$ we have that 
$$f(t)=-\frac{2\alpha (-\sin(t))^\alpha \sin\left(\frac{\pi}{2} \alpha\right)}{1-2b(t) \cos \left(\frac{\pi}{2} \alpha\right)+b(t)^2}\cdot \frac{1}{(1-\cos(t))^\alpha\sin(t)}\sim \frac{1}{(-\sin(t))^{1-\alpha}},\quad t\sim\pi^{+}.$$ Therefore $f(t)\to\infty$ when $t\to\pi^+$. Analogously, one can proved that $f(t)\to\infty$ when $t\to2\pi^-$. Hence, the minimum value of $f(t)$ is attained for $t\in(0,\pi)\cup(\pi,2\pi)$. From equation (\ref{def-b}), its follows that the relationship between $b(t)$ and $\sin(t)$ is given by \begin{equation}\label{relation-b-sin}\sin(t)=\dfrac{2 b(t)^{\frac{1}{\alpha}}}{1+b(t)^{\frac{2}{\alpha}}},\quad t\in(0,\pi);\quad \quad 
\sin(t)=\dfrac{-2 (-b(t))^{\frac{1}{\alpha}}}{1+(-b(t))^{\frac{2}{\alpha}}}, \quad t\in(\pi,2\pi).\end{equation} $b(t)^{1/\alpha}\sin(t)=1+\cos(t)$. Thus,  for all $t\in(0,\pi)$ or $t\in(\pi,2\pi)$ we obtain $$b^{\frac{1}{a}}\sin{t}= 1\pm\sqrt{1-\sin{t}^{2}},$$ which implies that $b^{\frac{2}{\alpha}}\sin^2(t)-2b^{\frac{1}{\alpha}}\sin(t)=-\sin^2(t)$, therefore $$\sin(t)=\frac{2 b(t)^{\frac{1}{\alpha}}}{1+b(t)^{\frac{2}{\alpha}}}.$$ Now, using the first part of (\ref{relation-b-sin}) we can rewrite $f(t)$ for all $t\in(0,\pi)$ as:
\begin{equation}\label{def-f(b)}
g(b)=f(t(b))=\frac{2\alpha(1+b^{\frac{2}{\alpha}})\sin{(\frac{\pi}{2}\alpha)}}{b^{\frac{1}{\alpha}-1} + 2b^{\frac{1}{\alpha}}\cos{(\frac{\pi}{2}\alpha)}+ b^{1+\frac{1}{\alpha}}}:[0,\infty)\to \R^+.\end{equation} Thus, the minimum value of $f(t)$ in $t\in(0,\pi)$ is the minimum value of $f(b)$ in $b\in(0,\infty)$. We shall show that $$\min_{b\in(0,\infty)}f(b)=f(1)=\dfrac{4\sin\left(\dfrac{\pi}{2}\alpha\right)}{1+\cos\left(\dfrac{\pi}{2}\alpha\right)}.$$ Consider $x=b^{1/\alpha}$, then $x$ left in the same interval. Thus, the corresponding function $g(x)=f(x^\alpha)/\alpha\sin(\alpha\pi/2)$ satisfies that $g:(0,\infty)\to \R$ and $$g(x)=\dfrac{1+x^2}{x^{1-\alpha}+2x\cos(\alpha\pi/2)+x^{1+\alpha}},\quad \alpha\in[0,1].$$ Thus, a straightforward calculations show that $$g'(x)=\dfrac{(1+\alpha)x^{\alpha}(x^{2-2\alpha}-1)+2\cos(\alpha\pi/2)(x^2-1)+(1-\alpha)x^{-\alpha}(x^{2+2\alpha}-1)}{(x^{1-\alpha}+\cos(\alpha\pi/2)x+x^{1+\alpha})^2},$$ which is positive when $x>1$, since every terms in the numerator of $g'(x)$ is positive, and is negative when $x<1$, for same reason as before. However, one can see that $g'(1)=0$ which implies that $g$ attained its minimum value at $x=1$. Thus, the minimum value of $f$ is $\alpha\sin(\alpha\pi/2)g(1)=f(1)$. 

Now, for values $t\in (\pi,2\pi)$ by the symmetry of lens-shape, the arguments as holds.
\end{proof}

\end{document}